\documentclass[12pt,a4paper]{amsart}
\usepackage{graphicx,amssymb}
\input xy
\xyoption{all}
\usepackage[all]{xy}
\usepackage{hyperref}

\newtheorem{thm}{Theorem}[section]
\newtheorem{cor}[thm]{Corollary}
\newtheorem{lem}[thm]{Lemma}
\newtheorem{prop}[thm]{Proposition}
\theoremstyle{definition}

\numberwithin{equation}{section}

\newcommand{\QQ}{\mathbb Q}

\newcommand{\ZZ}{\mathbb Z}
\newcommand{\CC}{\mathbb C}

\newcommand{\ra}{\rightarrow}

\newcommand{\cH}{\mathcal{H}}

 \DeclareMathOperator{\Ker}{Ker}

 \DeclareMathOperator{\Ima}{{Im}}

 \DeclareMathOperator{\tr}{tr}

\DeclareMathOperator{\End}{{End}}

\begin{document}

\title[Equations for abelian subvarieties ]{ Equations for abelian subvarieties}
\author{Angel Carocca, Herbert Lange and Rub\'{\i} E. Rodr\'iguez}
\address{Departamento de Matem\'atica y Estad\'istica, Universidad de La Frontera, Avenida Francisco Salazar 01145, Casilla 54-D, Temuco, Chile.}
\email{angel.carocca@ufrontera.cl}
\address{Department Mathematik, Universit\"at Erlangen, Cauerstrasse 11, 91058 Erlangen, Germany.}
\email{lange@math.fau.de}
\address{Departamento de Matem\'atica y Estad\'istica, Universidad de La Frontera, Avenida Francisco Salazar 01145, Casilla 54-D, Temuco, Chile.}
\email{rubi.rodriguez@ufrontera.cl}

\thanks{The  authors were partially supported by grants CONICYT PAI Atracci\'on de Capital Humano Avanzado del Extranjero  PAI80160004 and Anillo ACT 1415 PIA-CONICYT}

\date{\today }
\begin{abstract}
Given a finite group $G$ and an abelian variety $A$ acted on by $G$, to any subgroup $H$ of $G$, we associate an
abelian subvariety $A_H$ on which the associated  Hecke algebra $\cH_H$ for $H$ in $G$ acts. Any 
irreducible rational representation $\widetilde W$ of $\cH_H$ induces an abelian subvariety of $A_H$ in a natural way.
In this paper we give equations for this abelian subvariety. In a special case these equations become much easier. We work out some examples.

\end{abstract}

\maketitle

\section{Introduction}

Let $A$ be a complex abelian variety acted on by a finite group $G$. This induces an algebra homomorphism of the rational 
group ring $\QQ[G]$  into the rational endomorphism ring $\End_\QQ(A) = \End(A) \otimes \QQ$,
$$
\rho: \QQ[G] \ra \End_\QQ(A).
$$
For any element $\alpha \in \QQ[G]$ we define its  image in $A$ by
$$
\Ima(\alpha) := \Ima(n \alpha)
$$
where $n$ is any positive integer such that $n \alpha$ is an endomorphism. Since multiplication by a non-zero integer on $A$ is an isogeny, the definition does not depend on the chosen integer $n$.

Now for any subgroup $H$ of $G$ the element $p_H := \frac{1}{|H|} \sum_{h \in H} h$ is an idempotent in $\QQ[G]$ which defines an abelian subvariety of $A$,
$$
A_H := \Ima(p_H).
$$ 
The action of $G$ on $A$ induces an action of the Hecke algebra $\cH_H$ on the abelian subvariety $A_H$. The aim of this 
note is to study this action and use it to find equations for the abelian subvarieties $A_{H,\widetilde W}$ of  $A_H$ which are  given by the rational representations $\widetilde W$ of $\cH_H$.  Here \lq \lq equation\rq\rq means 
to express $A_H$ as the connected component containing 0 of the zero-set of an endomorphism of $A_H$. Since this is fairly
complicated, we do not repeat the result here, but refer to Theorem \ref{thm1} below. However, we get an easy and important consequence: If the group $G$ acts on an abelian variety $A$ and $H$ is any subgroup of $G$, then Corollary 
\ref{cor3.3} describes the complement of the  abelian  subvariety $A_G$ in $A_H$. This generalizes  \cite[Corollary 3.5, p.10]{rr} where it is proven for the Prym varieties $P(C/H, C/G)$ for any curve $C$  acted on by $G$ and any subgroup
$H$ of $G$.

In a special case the result becomes much simpler, namely let $V$ be an irreducible complex representation with
$\dim V^H =1$ and rational character. To this a one-dimensional rational representation $\widetilde W$ of the Hecke algebra 
$\cH_H$  is associated in a natural way and one can find an explicit basis $q_1, \dots ,q_s$ of the algebra $\cH_H$ for which we have (see Theorem \ref{thm4.3}) for the isotypical component $A_{H,\widetilde W}$ corresponding to 
$\widetilde W$ in $A_H$,

\begin{thm} 
With these assumptions the abelian subvariety $A_{H,\widetilde W}$ is 
$$
A_{H,\widetilde W} = \{z \in A_H \;|\;  q_i(z) = \chi_V(q_i) z \ \textup{ for } 1 \leq i \leq  s \}_0 \, ,
$$
where the index $0$ mean the connected component containing $0$.
\end{thm}
Several examples for this theorem will be given.\\

Section 2 contains some preliminaries. In particular we recall from \cite{caro} the decompositions of $A_H$ induced by the 
Hecke algebra $\cH_H$. Section 3 contains the proof of the general theorem (Theorem \ref{thm1}). In Section 4 we prove the above mentioned theorem. Finally Section 5 contains some examples.

\section{Preliminaries} 

Let $G$ be a finite group acting on an abelian variety $A$ of dimension $g$ over the field of complex numbers and let 
$H \leq G$ be a subgroup of $G$.
The Hecke algebra for $H$ in $G$ acts in a natural way on the abelian variety $A$. To be more precise: the element 
$$
p_H := \frac{1}{|H|} \sum_{h \in H} h
$$ 
is an idempotent in the group algebra $\QQ[G]$.

The {\it (rational)  Hecke algebra for $H$ in $G$} is defined as the subalgebra
$$
\cH_H := p_H\QQ[G ]p_H = \QQ [H\backslash G/H]
$$
of the rational group group algebra $\QQ[G]$. If we consider $\QQ[G]$ as the algebra of functions $G \ra \QQ$ with 
multiplication the convolution product (see \cite[\S 11]{cr}), then $\QQ[H\backslash G/H]$ is the subalgebra of functions which are constant
on each double coset $HgH$, and $p_H$ is the unit in this algebra.

The action of $G$ on $A$ induces an action of $\QQ[G]$ on $A$ in a natural way, giving an algebra homomorphism
\begin{equation} \label{eq:homo}
\QQ[G] \ra \End_\QQ(A) \, .
\end{equation}
Since this homomorphism is canonical, we denote the elements of $\QQ[G]$ and their images by the same letter. For any element
$\alpha \in \QQ[G]$ we define its image in $A$ by 
$$
\Ima(\alpha) := \Ima(n \alpha) \subset A
$$ 
where $n$ is any positive integer such that $n\alpha$ is in $\End(A)$. It is an abelian subvariety of $A$ which does not depend on 
the chosen integer $n$. 

Consider the abelian subvariety of $A$ given by
$$
A_H := \Ima(p_H).
$$
Restricting \eqref{eq:homo} to $\cH_H$ gives an algebra
homomorphism 
\begin{equation} \label{eq:homo2}
\cH_H \ra \End_\QQ(A_H).
\end{equation}

The aim of this section is to recall from \cite{caro} the isotypical and Hecke algebra decompositions of $A_H$ with respect to this action of $\cH_H$.

Let $\{W_1, \dots, W_r\}$ denote the irreducible rational representations of $G$.
To any $W_i$ there corresponds an
 irreducible complex representations $V_i$, uniquely determined up to an element of the Galois group of $K_i$ over $\QQ$,
where $K_i$ is the field obtained by adjoining to $\QQ$ the values of the character $\chi_{V_i}$  of $V_i$. The representations
$W_i$ and $V_i$ are said to be  {\it Galois associated}.

To each $W_i$ we can associate a central idempotent $e_{W_i}$ of $\QQ[G]$ by
$$
e_{W_i} = \frac{\dim_\CC(V_i)}{|G|} \sum_{g \in G}  \tr_{K_i/\QQ}(\chi_{V_i}(g^{-1}))g \, .
$$
Let $\rho_H$ denote the representation of $G$ induced by the trivial representation of $H$. It decomposes as
\begin{equation} \label{e2.1}
\rho_H \simeq \sum_{i=1}^r a_i W_i,
\end{equation}
with $a_i = \frac{1}{s_i} \dim_\CC(V_i^H)$ and $s_i$ the Schur index of $V_i$.  Renumbering if necessary, let $\{W_i, \dots, W_t\}$ denote the set of all 
irreducible  rational representations of $G$ such that $a_i \neq 0$. Then there is a bijection  from this set to the set 
$\{ \widetilde W_1, \dots, \widetilde W_t\}$ of all irreducible rational representations of the algebra $\cH_H$.
An analogous statement holds for the irreducible complex representations of $G$ and of $\CC [H\backslash G/H]$. Let $\widetilde V$ denote the  
representation of  $\CC [H\backslash G/H]$ associated to the complex irreducible representation $V$ of $G$.

According to \cite[equation (2.4)]{caro} and \cite[p. 331]{caro0} the dimension of $\widetilde W_i$ given by
\begin{equation} \label{eq2.2}
\dim_\QQ(\widetilde W_i) = \dim_\QQ (W_i^H) = [L_i:\QQ] \dim_\CC(V_i^H).
\end{equation}
where $L_i$ denotes the field of definition of the representation $V_i$. Recall that the index $s_i=[L_i:K_i]$ is the Schur index of $V_i$.

For $i = 1, \dots,t$ consider the central idempotents of $\cH_H$, 
$$
f_{H,\widetilde W_i} := p_H e_{W_i} = e_{W_i} p_H.
$$ 
Then $p_H$ decomposes as 
$$
p_H = \sum_{i=1}^t f_{H,\widetilde W_i} \, .
$$
Defining for $i=1, \dots , t$ the abelian subvarieties 
$$
A_{H, \widetilde W_i} := \Ima(f_{H,\widetilde W_i}),
$$
one obtains the following isogeny decomposition of $A_H$,  given by the addition map
\begin{equation} \label{e2.2}
+:   A_{H, \widetilde W_1} \times A_{H, \widetilde W_2} \times \cdots \times A_{H, \widetilde W_t}  \ra A_H.
\end{equation}
It is uniquely determined by $H$ and the action of $G$ and  called the {\it isotypical decomposition of} $A_H$.

If $a_i \ge 2$ in \eqref{e2.1}, the subvarieties $A_{H,\widetilde W_i}$ can be decomposed further. In fact, given 
$\widetilde W_i$,
explicit  orthogonal primitive idempotents $f_{i,j}, 1 \le j \le n_i:= \frac{1}{s_i} \dim V_i$  may be found such that 
$$
e_{W_i} = f_{i,1} + \cdots + f_{i,n_i}.
$$
Multiplying by $p_H$ gives
$$
f_{H,\widetilde W_i} = f_{i,1}p_H + \cdots + f_{i,n_i}p_H.
$$
We label the $f_{i,j}p_H$ in such a way that for the first $a_i = \frac{1}{s_i} \dim (V_i^H)$ of them, the minimal left ideals
of the simple algebra $\QQ[G]f_{H,\widetilde W_i}$
$$
J_{i,j} := \QQ[G]f_{i,j}p_H \qquad \mbox {are different from} \; 0 \, ,
$$
and different among themselves; that is, such that
\begin{displaymath}
\QQ[G]f_{H,\widetilde W_i} =  \bigoplus_{j=1}^{a_i} J_{i,j} \, .
\end{displaymath}

Then there exist primitive idempotents $v_{i,1}, \dots, v_{i,a_i}$ in $\QQ[G]f_{H,\widetilde W_i}$, each $v_{i,j}$ generating the corresponding ideal 
$J_{i,j}$,  such that 
\begin{equation}  \label{e2.3}
f_{H,\widetilde W_i} = v_{i,1} + \cdots + v_{i,a_i}.
\end{equation}
Note that by construction the $v_{i,j}$ are orthogonal  idempotents, not uniquely determined, since the $f_{i,j}$ are not uniquely determined. Defining for $j = 1, \dots, a_i$,
$$
B_{H,\widetilde W_i,j} := \Ima (v_{i,j}),
$$
equation \eqref{e2.3} induces the following isogeny
$$
+: B_{H,\widetilde W_1,j} \times B_{H,\widetilde W_2,j} \times \cdots \times B_{H,\widetilde W_{a_i},j} \ra A_{H, \widetilde W_i}.
$$
Here the subvarieties $B_{H,\widetilde W_1,j}, \dots,  B_{H,\widetilde W_{a_i},j}$ are pairwise isogenous. Hence combining 
with the isogeny \eqref{e2.2}, we get the following $\cH_H$-equivariant isogeny
$$
B_{H,\widetilde W_1,1}^{a_1} \times B_{H,\widetilde W_2,1}^{a_2} \times \cdots \times B_{H,\widetilde W_t,1}^{a_t} \ra A_H \, ,
$$
called the {\it Hecke algebra decomposition of} $A_H$ with respect to the action of $G$ and the subgroup $H$.

\section{equations for the abelian subvarieties $ A_{H, \widetilde W_i}$}

In this section we describe the abelian subvarieties $ A_{H, \widetilde W_i}$. The same method works for the abelian
subvarieties $B_{H,\widetilde W_i,j}$, which we omit however. The method relies of the following proposition.

Let $G$ be a finite group acting on the abelian variety $A$.
For any idempotent $\iota$ of $\QQ[G]$ there are two abelian subvarieties of $A$, namely
$$
A_\iota := \Ima{\iota} \subset A \qquad \mbox{and} \qquad C_\iota:= \Ima (1_G - \iota).
$$

\begin{prop} \label{p3.1}
\begin{enumerate}
\item The addition map gives an isogeny
$$
+: A_\iota \times C_\iota \ra A.
$$
\item 
\begin{center}
$A_\iota  = \Ker (1_G-\iota)_0 \qquad \mbox{and} \qquad C_\iota = \Ker (\iota)_0.$
\end{center}
where the index $0$ means the connected component containing $0$.
\end{enumerate}
\end{prop}

To be more precise, for an element $\alpha \in \End_\QQ(A)$ we denote by $\Ker (\alpha)_0$ the connected component 
containing 0
of some positive multiple $n \alpha$ which is an endomorphism. This does not depend on the chosen $n$. 

Given a polarization of $A$, there is an analogous result for any abelian subvariety using norm-endomorphisms 
(see \cite[Section 5.3]{bl}). One could give also a proof of the proposition introducing polarizations, but we prefer to give a direct proof.

\begin{proof}
(1): Since $\iota + (1_G - \iota) = 1_G$, the addition map $+: A_\iota \times C_\iota \ra A$ is surjective. To see that is is an 
isogeny, it suffices to show that on the level of tangent spaces it is an isomorphism.  So suppose $A = V/\Lambda$ with a 
$\CC$-vector space $V$ and a lattice $\Lambda$. we may choose the basis of $V$ such that the analytic representation
$\rho_a(\iota)$ is given by the diagonal matrix diag$(1,\dots,1,0,\dots,0)$ with the number of $1$'s equal to $\dim A_\iota$. 
Hence $\rho_a(1_G - \iota) =$ diag$(0, \dots,0,1,\dots 1)$ with the number of $0$'s equal to $\dim A_\iota$. This implies
$\dim C_\iota = \dim A - \dim A_\iota$ and thus the assertion.

(2): Choose a positive integer $n$ such that $n(1_G - \iota) \in \End_\QQ(A)$ and consider the exact sequence 
$$
0 \ra \Ker(n(1_G - \iota)) \ra A \stackrel{n(1_G - \iota)}{\ra} C_\iota \ra 0.
$$
Certainly 
$A_\iota \subset \Ker(n(1_G - \iota))$. Since both have the same dimension, this implies the assertion.
\end{proof}

Now let $H$ be a subgroup of $G$ and $\iota$ be an idempotent of the corresponding Hecke algebra $\cH_H$. It induces two 
abelian subvarieties of the abelian variety $A_H$, namely
$$
A_{H,\iota} := \Ima(\iota) \qquad \mbox{and} \qquad P_{H,\iota} := \Ima (p_H - \iota).
$$
Notice that $p_H - \iota: A_H \ra A_H$ is also an idempotent of $\cH_H$, since $p_H$ is the unit element of the algebra 
$\cH_H$ and hence its image in  $\End_\QQ(A_H)$  is the identity on $A_H$.  If we denote by $\iota$ also the corresponding idempotent of $\End_\QQ(A_H)$, this implies that $p_{H} -\iota$ is 
an idempotent of $\End_\QQ(A_H)$.

\begin{prop} \label{p3.2}
\begin{enumerate}
\item The addition map gives an isogeny
$$
+: A_{H,\iota} \times C_{H,\iota} \ra A_H.
$$
\item 
\begin{center}
$A_{H,\iota}  = \Ker (p_H - \iota)_0 \qquad \mbox{and} \qquad P_{H,\iota} = \Ker (\iota)_0.$
\end{center}
\end{enumerate}
\end{prop}
\noindent
The proof is essentially the same as the proof of Proposition \ref{p3.1}. 
We want to use the proposition to describe the isotypical components 
of $A_H$ as fixed-point sets of particular endomorphisms of $A_H$.

First we consider the trivial representation $W_0$ of $G$ and any subgroup $H$ of $G$.
Note that 
$$
e_{W_0} := p_G = \dfrac{1}{|G|} \sum_{g \in G} g
$$
 is the central idempotent in $\mathbb{Q}[G]$ corresponding to $W_0$.
Moreover, choose  a set of representatives $g_1 , \ldots , g_s$ for both the right and left cosets of $H$ in $G$. Such a set exists 
according to \cite[Theorem 5.1.7]{h}.  Using this, we get as a special case,

\begin{cor}  \label{cor3.3}
With the above notations we have for the abelian subvariety $A_G = (A^G)_0$ and the complement $P(A_H, A_G)$
of  $A_G$ in $A_H = \Ima(p_H)$,
\begin{enumerate}
\item[(i)] $A_G = \{ z \in A_H :  \sum_{j=1}^s g_j(z) =z  \}_0,$
\item[(ii)]  $P(A_H, A_G) = \{  z \in A_H : \sum_{j=1}^s g_j(z) =0  \}_0.$
\end{enumerate}
\end{cor}
 
The notation $P(A_H,A_G)$ comes from the fact that in the case of a Galois cover of curves  $ C \ra  C/G$ and any subgroup $H$ of $G$ this is just the Prym variety $P( C/H,  C/G)$. 
In fact, Corollary \ref{cor3.3} generalizes \cite[Corollary 3.5, p.10]{rr} where it is proven for these Prym varieties 
$P(C/H, C/G)$.

\begin{proof}
We may assume that  $H \lneq G$. Then we have
	\begin{equation} \label{e33}
	p_H = (p_H - e_{W_0}) + e_{W_0}.
	\end{equation}	
 Observe that because of the special choice of the $g_i$,
	$$
	[G:H] e_{W_0}  = p_H \left( \sum_{j=1}^s g_s \right) = \left( \sum_{j=1}^s g_j \right) p_H.
	$$
Hence $e_{W_0}$ is  in $\cH_H$ and furthermore $\Ima(e_{W_0}) = A_G$.
The assertion  now follows from Proposition \ref{p3.2}, since
	$$
	A_G = \Ker(p_H -  e_{W_0})_0 = \{ z \in A_H :  \sum_{j=1}^s g_j(z) =z  \}_0 
	$$
	and 
	$$
	P(A_H, A_G) =  \Ker(e_{W_0})_0 = \{  z \in A_H : \sum_{j=1}^s g_j(z) =0  \}_0
	$$
and $A_H = (A^H)_0$.
\end{proof}

Let $\widetilde W$ be the irreducible rational representation of the Hecke algebra $\cH_H$ associated to the irreducible rational 
representation $W$ of $G$ and 
$$
f_{H,\widetilde W} = p_H e_W = e_W p_H
$$ 
the corresponding central idempotent of $\cH_H$. In order to find a more convenient expression of $f_{H,\widetilde W}$,
consider the decomposition of $G$ into double cosets of $H$ in $G$,
$$
G = H_1 \cup H_2 \cup \cdots \cup H_s
$$
with $H_i = Hx_iH$ and $x_1 =1$, i.e. $H_1 =H$.
A basis for the Hecke algebra $\cH_H$ is given by the elements
\begin{equation} \label{e3.1}
q_i:= \frac{|H|}{|H \cap x_iHx_i^{-1}|} p_H x_i p_H = \frac{1}{|H|} \sum_{h_i \in H_i} h_i
\end{equation}
for $i = 1, \dots, s$  (see \cite[Proposition 11.30(i)]{cr}). For the last equation use that
\begin{equation} \label{eq3.1}
\frac{|H|}{|H \cap x_iHx_i^{-1}|} =|H: H \cap  x_iHx_i^{-1}| = \frac{|Hx_iH|}{|H|} = \frac{|H_i|}{|H|}
\end{equation}

Let $\{g_{i,j}, \; j = 1, \dots, d_i:= |H:x_iHx_i^{-1}|\}$ denote a set of simultaneous representatives for the right and left cosets for $H$ in $H_i$. Such a set exists again by \cite[Theorem 5.1.7]{h}.  Then we have,
\begin{lem} \label{l3.3}
Considering the elements $q_i$ as elements of $\End_\QQ(A_H)$, we have: $q_i$ is an endomorphism of $A_H$ and as such
$$
q_i (z) = \sum_{j=1}^{d_i} g_{i,j}(z) 
$$
for each $z \in A_H$.
\end{lem}

\begin{proof}
Since $H_i = \cup_{j=1}^{d_i} Hg_{i,j} =  \cup_{j=1}^{d_i} g_{i,j}H$, we have,
$$
q_i =  \frac{1}{|H|} \sum_{x \in H_i} x = p_H \left( \sum_{j=1}^{d_i} g_{i,j} \right) = \left( \sum_{j=1}^{d_i} g_{i,j} \right) p_H.
$$
This gives the assertion, since $p_H$ is the identity on $A_H$.
\end{proof}

\begin{lem} \label{lem3.3}
The following equality is valid in the Hecke algebra $\cH_H$:
$$
f_{H, \widetilde W} = \frac{\dim V}{|G|} \sum_{i=1}^s \left( |H \cap x_iHx_i^{-1}|   \tr_{K/\QQ}(\chi_V(q_i)) \right) q_i
$$
\end{lem}

\begin{proof}Since $e_W$ is central and $p_H^2 = p_H$, we have 
\begin{eqnarray*}
f_{H, \widetilde W} = p_H e_W &=& p_H e_W p_H = \frac{\dim V}{|G|} \sum_{g \in G} \tr_{K/\QQ}(\chi_V(g^{-1})) (p_H g p_H)\\ 
&=& \frac{\dim V}{|G|} \sum_{i=1}^s  \sum_{g \in H_i}  \tr_{K/\QQ}(\chi_V(g^{-1}))   p_H g p_H\\
&=& \frac{\dim V}{|G|}  \sum_{i=1}^s \left( \sum_{g \in H_i}  \tr_{K/\QQ}(\chi_V(g^{-1}))   p_H x_i p_H \right)\\
&=&  \frac{\dim V}{|G|}  \sum_{i=1}^s \left( \sum_{g \in H_i} \frac{|H \cap x_iHx_i^{-1}|}{|H|}  \tr_{K/\QQ}(\chi_V(g^{-1})) \right) \frac{|H|}{|H \cap x_iHx_i^{-1}|}  p_H x_i p_H.\\
\end{eqnarray*}

But for any $g \in G$ we have $\chi_V(g^{-1}) = \overline{\chi_V(g)}$, and therefore 
$$
\sum_{g \in H_i} \tr_{K/\QQ}(\chi_V(g^{-1})) = \sum_{g \in H_i} \tr_{K/\QQ}(\chi_V(g)) = |H| \tr_{K/\QQ}(\chi_V(q_i)).
$$

So equation \eqref{e3.1} gives the assertion.
\end{proof}

Recall that we consider any element of $\QQ[G]$ also as en element of $\End_\QQ(A)$. 
We denote for any element $\alpha$ of $\cH_H \subset \QQ[G]$ by 
$$
\Ker (\alpha)_0
$$ 
the connected component of $\Ker (n\alpha) \subset A_H$ containing $0$,  where $n$ is any positive integer such that
 $n \alpha$  is actually an endomorphism. This does not depend of the chosen $n$. Then we get
the following equation for the abelian subvariety $A_{H, \widetilde W}$ of $A_H$.

\begin{thm} \label{thm1}
The isotypical component $A_{H,\widetilde W}$ of $A_H$ is given by 
$$
A_{H,\widetilde W} = \Ker\left( p_H - \frac{ \dim V}{|G|}  \sum_{i=1}^s \left(  |H \cap x_iHx_i^{-1}|  \tr_{K/\QQ}(\chi_V(q_i)) \right) q_i  \right)_0.
$$
\end{thm}

\begin{proof}
By definition of $A_{H,\widetilde W}$ and Proposition \ref{p3.2} we have 
$$
A_{H,\widetilde W} = \Ima(f_{H,\widetilde W}) =\Ker (p_H - f_{H,\widetilde W})_0.
$$
So Lemma  \ref{lem3.3} gives the assertion.
\end{proof}

In the next section we will use the following orthogonality relations,

\begin{lem} \label{orth}
Let $\widetilde U, \widetilde V$ be complex irreducible representations of $\mathbb{C}[H\backslash G /H]$ associated to the irreducible representations 
$U, V$ of $G$. Then, with the notation of above, 
$$
\sum_{j=1}^s \frac{|H|}{|H \cap x_jHx_j^{-1}|} \chi_{\widetilde U}(q_j^{-1}) \chi_{\widetilde V}(q_j) 
= \left\{\begin{array}{lll}
             0 & if & \widetilde U \neq \widetilde V;\\
            \frac{[G:H]}{\dim U} & if & \widetilde U = \widetilde V,
           \end{array}        \right.
$$
where $q_j^{-1} := \frac{1}{|H|} \sum\limits_{y \in H_j} y^{-1}$.
\end{lem}

\begin{proof}
The orthogonality relations \cite[Theorem 11.32 (ii)]{cr} say, in the special case that $\psi$ is the trivial representation,
$$
\sum\limits_{j=1}^s \frac{1}{[H:H \cap x_jHx_j^{-1}]} \, \chi_{\widetilde{U}} \left( \frac{1}{|H|} \sum\limits_{y \in H_j} y^{-1} \right) \, \chi_{\widetilde{V}} \left( \frac{1}{|H|} \sum\limits_{x \in H_j} x \right) =
\left\{
\begin{array}{lll}
0, & if & \widetilde{U} \neq \widetilde{V}; \\
\frac{[G:H]}{\dim V}  & if & \widetilde{U} = \widetilde{V}.
\end{array}
\right.
$$
This implies the assertion. 
\end{proof}

The following lemma is proven in \cite[p.282, l. -4]{cr}.

\begin{lem} \label{l3.7}
Let $\widetilde V$ be a complex irreducible representation of $\mathbb{C}[H\backslash G /H]$ associated to the irreducible representation 
$V$ of $G$. Then their characters satisfy
$$
\chi_{\widetilde V}(x) =\chi_V(x) \quad \mbox{for all} \quad x \in \mathbb{C}[H\backslash G /H].
$$
\end{lem}

\section{A special case}

The equation of Theorem \ref{thm1} seems fairly complicated. In a special case we can describe the subvariety $A_{H, \widetilde{W}}$ in a simpler way.
Let the notation be as above, but assume in addition that
\begin{equation} \label{e4.1}
\dim V^H =1 \qquad \mbox{and} \qquad K = \QQ \ .
\end{equation}
According to \cite[Corollary 10.2]{i}, this implies that also the Schur index $s(V)$ of $V$ is equal to 1.  Then we have, according to equation  \eqref{eq2.2},
$$
\dim \widetilde W = \dim W^H = 1.
$$
This implies that the complex representation $\widetilde V$ of $\mathbb{C}[H\backslash G /H]$  is  rational of dimension $1$ with $\widetilde W = \widetilde V \otimes \CC$, and the complex representation $V$ of $G$ is rational with $ W =  V \otimes \CC$.

\begin{prop} \label{l4.1}
With the assumption \eqref{e4.1} we have
$$
f_{H, \widetilde W} =  \frac{\dim V}{|G|} \sum_{i=1}^s |H \cap x_iHx_i^{-1}| \chi_V(q_i) q_i.
$$
\end{prop}

\begin{proof}
This is a direct consequence of  Lemma \ref{lem3.3}.
\end{proof}

Now consider all elements of $\cH_H$ as elements of $\End_\QQ(A_H)$. According to Lemma \ref{l3.3}, $ q_i$ is an endomorphism of $A_H$
for all $i$ and we can express $A_{H,\widetilde W}$ as the kernel of an actual endomorphism. In fact, we get as a direct consequence of Proposition \ref{l4.1} and Theorem \ref{thm1},

\begin{cor} \label{c4.2}
With the assumption \eqref{e4.1} let $n:= |G|$. Then $n f_{H,\widetilde W}$ is an endomorphism of $A_H$   and we have
$$
A_{H,\widetilde W} = \Ker \left( n_{A_H} - \dim V \sum_{i=1}^s |H \cap x_iHx_i^{-1}| \chi_V(q_i) q_i \right)_0.
$$
\end{cor} 

Define the abelian subvariety $B_{H,\widetilde W}$ by
$$
B_{H,\widetilde W}:= \{ z \in A_H \;|\; q_i(z) = \chi_V (q_i) z \; \mbox{for} \; i = 2, \dots , s \}_0.
$$
Note that the condition $q_1(z) = \chi_V(q_1)z \;\; (= z)$ is hidden the the assumption $z \in A_H$. So one could equivalently write
$$
B_{H,\widetilde W}= \{ z \in A \;|\; q_i(z) = \chi_V (q_i) z \; \mbox{for} \; i = 1, \dots , s \}_0.
$$
The aim of this section is the proof of the following theorem.
\begin{thm} \label{thm4.3}
Under the assumptions \eqref{e4.1} we have
$$
A_{H,\widetilde W} =B_{H,\widetilde W}.
$$
\end{thm}

Recall that $\{f_{H,\widetilde W}\}_{\widetilde W}$ are the central primitive idempotents in $\cH_H$ and the isotypical decomposition 
of $\cH_H$ is 
$$
\cH_H = \oplus_{\widetilde W} \cH_H f_{H,\widetilde W} \ ,
$$
where $\widetilde W$ acts on the simple subalgebra $\cH_H f_{H,\widetilde W}$ and by $0$ on the other components. Recall moreover that the $q_i, \; i = 1, \dots,s$ 
as defined in \eqref{e3.1} are a basis of $\cH_H$.
\begin{lem} \label{l4.4}
Under the assumption \eqref{e4.1} the action of the Hecke algebra $\cH_H$ on the abelian variety $A_{H,\widetilde W}$ is given by 
$$
q_i(z) = \chi_V(q_i) z \quad \textup{for all}  \;\;  z \in A_{H,\widetilde W} \; \textup{ and for all } \; i = 1, \dots , s.
$$
\end{lem}

\begin{proof} 
Note first that the left hand side of the equation makes sense, since $q_i$ is an endomorphism on $A_H$ by Lemma \ref{l3.3}.
In order to see that also the right hand side makes sense, we have to show that  $\chi_V(q_i)$ is an integer. But \cite[Lemma 7.1]{in} says that for any subgroup $H$ of any finite group $G$ and any complex representation $V$ of $G$ the numbers $\chi_V(q_i)$
are algebraic integers. Since in our case $\chi_V$ has rational values, this means $\chi_V(q_i)$ is an integer.

For the proof of the lemma, it suffices to show that the analogous equation is valid for the action of $\cH_H$ on $\cH_H f_{H,\widetilde W}$ by $\widetilde W$,
 i.e. to show
$$
q_i(x) = \chi_V(q_i) x \quad \mbox{for all} \;\; x \in \cH_H f_{H,\widetilde W}.
$$
Since $\widetilde W$ is of dimension one,  we have $\widetilde V = \widetilde W \otimes \CC$ and  $\cH_H f_{H,\widetilde W}$ is a simple algebra of dimension $1$, hence equal to  $\cH_H f_{H,\widetilde W} = \mathbb{Q} f_{H,\widetilde W}$. But the action of a one-dimensional 
complex representation is given by multiplication by the character (which equals the representation). 
This implies in particular $q_i(x) = \chi_{\widetilde V}(q_i)x$ for all i and all $x \in \cH_H f_{H,\widetilde W}$. Since  $\chi_{\widetilde V} = \chi_V$
on $\cH_H$ by Lemma \ref{l3.7}, this gives the assertion. This completes the proof of the theorem.
\end{proof}

\begin{proof}[Proof of Theorem \ref{thm4.3}]
Recall that
$$
A_{H,\widetilde W} = \Ima(|G|f_{H,\widetilde W}) = \{ z \in A_H \;|\; |G|f_{H,\widetilde W}(z) = |G|z \}_0
$$
First we show $A_{H,\widetilde W} \subset B_{H,\widetilde W}$:  Suppose that $z \in A_{H,\widetilde W}$. Since $q_i = q_if_{H,\widetilde W}$, we get for ail $i$,
\begin{eqnarray*}
|G| q_i(z) &=& |G|q_i f_{H,\widetilde W}(z)\\
&=& \chi_V(q_i) |G| f_{H,\widetilde W}(z) \quad (\mbox{by Lemma}\; \ref{l4.4})\\
&=& \chi_V(q_i) |G| (z).
\end{eqnarray*}
So 
$$
z \in \{z\in A_H \;|\; |G|q_i(z) = \chi_V (q_i) |G| z\}_0 =  \{ z \in A_H \;|\; q_i(z) = \chi_V(q_i) z \}_0
$$
i.e. $z \in B_{H,\widetilde W}$.

Finally we show  $B_{H,\widetilde W} \subset A_{H,\widetilde W}$: So suppose $z \in B_{H,\widetilde W}$, i.e. 
$$
q_i(z) = \chi_V(q_i) z   \quad \mbox{for all} \;\; i = 1, \dots,s.
$$
Since in our case $\chi_V(g^{-1}) = \chi_V(g)$ for all $g$ and $\chi_{\widetilde V} = \chi_V|_{\cH_H}$ by Lemma \ref{l3.7},
part of the orthogonality relations \ref{orth} become
$$
\dim V\sum_{i=1}^s  |H \cap x_iHx_i^{-1}| \chi_V(q_i)\chi_V(q_i) = |G|.
$$
From this we get 
\begin{eqnarray*}
\dim V \sum_{i=1}^s |H \cap x_i Hx_i^{-1}| \chi_V(q_i) q_i(z) &=& \dim V \sum_{i=1}^s |H \cap x_i Hx_i^{-1}| \chi_V(q_i) \chi_V(q_i)(z) \\
&=& |G| z.
\end{eqnarray*}
This means that $z \in A_{H,\widetilde W}$.
\end{proof}

\section{Examples}

In this section we work out the equations of Theorem \ref{thm4.3} in some cases. For the computations we used the computer 
program GAP.

\subsection{Example 1} Let $G = S_4$ and $W = V$ the standard representation of degree 3: 
$$
\chi_W (1) = 3, \chi_W(..) = 1, \chi_W((..)(..)) = -1, \chi_W(...) =0, \chi_W(....) = -1
$$
where $n$ dots mean any cycle of length $n$.\\

{\bf Case 1}: $H = \langle (12),(34) \rangle \simeq \ZZ_2 \times \ZZ_2$ (not normal in $G$).

Clearly the assumptions \eqref{e4.1} are satisfied and 
$$
G = H \cup (H(23)H) \cup (H(13)(24)H) =: H_1 \cup H_2 \cup H_3.
$$
One checks
$$
q_1 = \frac{1}{4} \sum_{h \in H_1} h = p_H\; ,\quad q_2 = \frac{1}{4}\sum_{k \in H_2} k \;, \quad  q_3 = \frac{1}{4} \sum_{k \in H_3} k
$$
and  
$$
\chi_W(q_1) = 1, \quad \chi_W(q_2) = 0, \quad \chi_W(q_3) = -1.
$$
We get from Theorem \ref{thm4.3} for any abelian variety $A$ with $G$-action,
$$
A_{H,\widetilde W} = \{ z \in A_H \;:\; q_2(z) =0,\; q_3(z) = -z \}_0.
$$
 
Now observe that, 
$$
H_3 = \{(13)(24), (14)(23),(1324),(1423)\} = D_4 - H
$$
where $D_4$ denotes the dihedral subgroup of $S_4$. This implies 
$$
A_{\ZZ_2 \times \ZZ_2,\widetilde W} = P(A_{\ZZ_2 \times \ZZ_2}/A_{D_4})
$$
where $P(A_{\ZZ_2 \times \ZZ_2}/A_{D_4})$ denotes the complement of $A_{D_4}$ in $A_{\ZZ_2 \times \ZZ_2}$.

\medskip

{\bf Case 2}: $H = \langle (34), (243) \rangle \simeq S_3$: 

Again the assumptions \eqref{e4.1} are satisfied. Here we have 
$$
G = H \cup (H(14)H) =: H_1 \cup H_2.
$$
One checks 
$$
q_1 = \frac{1}{6} \sum_{h \in H} h = p_H \;, \qquad q_2 = \frac{1}{6} \sum_{k \in H_2}k
$$
and 
$$
\chi_W(q_1) = 1 \; , \qquad \chi_W(q_2) = -1.
$$
We get from Theorem \ref{thm4.3} for any abelian variety $A$ with $G$-action,
$$
A_{H,\widetilde W} = \{ z \in A_H \;:\; q_2(z) = -z \}_0.
$$
Now observe that $H_2 = S_4 - H$, which implies, as in Case 1,
$$
A_{H,\widetilde W} = P(A_{S_3}/A_{S_4}).
$$

\subsection{Example 2} Let $G = N \rtimes P \simeq \ZZ_2^4 \rtimes \ZZ_5$ denote the 
subgroup of order 80 of $S_{10}$ (which occurred in 
\cite{clr}) that is generated by the permutations  $s_i := (i \;i+1)(5+i\; 6+i)$ for $i = 1,\dots,4$ and 
$\sigma := (1 \,2 \, 3 \, 4\, 5)(6 \, 7 \, 8 \, 9 \, 10)$. 

For $i = 1,2,3$ consider the rational irreducible representation 
$W_i$ defined as follows:
The group $P$ acts on the character group $\widehat N$ of $N$ in the usual way. Apart from the trivial representation $\chi_0$  there are 3 orbits
of the action of $P$ on $\widetilde N$. Let $\chi_i, \; i=1,2,3$
be representatives of them. Then, if $\rho$ denotes the  representation of $G$ induced by the trivial representation of $N$,
the rational irreducible representation $W_i$ of degree five is defined as
$$
W_i = \rho \otimes \chi_i \quad \mbox{for} \quad i = 1,2,3.
$$

The characters of $W_i$ for $i = 1,2,3$ are given by $\chi_{W_i}(1_G) = 5$, $\chi_{W_i}(\sigma) = 0$, and
$$
\chi_{W_1}(s_j) = 1, \;\;\chi_{W_1}(s_1s_2) = 1,\;\; \chi_{W_1}(s_1s_3) = -3,
$$
$$
\chi_{W_2}(s_j) = -3, \;\;\chi_{W_2}(s_1s_2) = 1,\;\; \chi_{W_2}(s_1s_3) = 1,
$$
$$
\chi_{W_3}(s_j) = 1, \;\;\chi_{W_3}(s_1s_2) = -3,\;\; \chi_{W_3}(s_1s_3) = 1.
$$

\medskip

{\bf Case 1}: $H = P$. \\
For the representation $\rho_H$ of $G$ induced by the trivial representation of $H$ we have
$\rho_H = \chi_0 \oplus W_1 \oplus W_2 \oplus W_3$ 
which implies for $i = 1,2,3$,
$$
\dim \widetilde W_i=\dim W_i^H = \langle W_i, \rho_H \rangle = 1.
$$
Hence the assumptions \eqref{e4.1} are satisfied for all $W_i$. The double coset decomposition of $G$ is
$$
G = H \cup (Hs_1H) \cup (Hs_1s_2H) \cup (Hs_1s_3H) =: H_1 \cup H_2 \cup H_3 \cup H_4.
$$
The double cosets $H_2,H_3,H_4$ contain each a complete conjugacy class of involutions of $G$ and all other elements are of order 
5. A basis of the (commutative) Hecke algebra $\cH_H$ is given by
$$
q_1 = p_H,\quad \mbox{and} \quad q_i = \frac{1}{5} \sum_{k_i \in H_i} k_i \quad \mbox{for} \;\; i = 2,3,4.
$$
The multiplication of $\cH_H$ is given by
$$
q_1q_j = q_j,\;\; q_2^2 = 5q_1 + 2q_3 + 2q_4,\;\; q_3^2 = 5q_1 +2q_2 +2q_4, \;\;q_4^2 = 5q_1 +2q_2 + 2q_3,
$$
$$
q_2q_3 = 2q_2+2q_3 +q_4, \;\; q_2q_4 = 2q_2 + q_3 + 2q_4, \;\; q_3q_4 = q_2 + 2q_3 + 2q_4.
$$


Since
$$
\chi_{W_i}(q_1) = \chi_{W_1}(q_2) = \chi_{W_1}(q_3) = \chi_{W_2}(q_3) =\chi_{W_2}(q_4) =\chi_{W_3}(q_2) =\chi_{W_3}(q_4) = 1 \ ,
$$
and
$$
\chi_{W_1}(q_4) = \chi_{W_2}(q_2) = \chi_{W_3}(q_3) = -3 \ ,
$$
we get from Theorem \ref{thm4.3} that for any abelian variety $A$  with an action of the group $G$, 
$$
A_{H,\widetilde W_1} = \{z \in A_H \;|\;  q_2(z) = q_3(z) = z, \; q_4(z) = -3z \}_0,
$$  
$$
A_{H,\widetilde W_2} = \{z \in A_H \;|\;  q_3(z) = q_4(z) = z, \; q_2(z) = -3z \}_0,
$$  
$$
A_{H,\widetilde W_3} = \{z \in A_H \;|\;  q_2(z) = q_4(z) = z, \; q_3(z) = -3z \}_0.
$$  

\medskip

{\bf Case 2}: $H =  \langle s_1,s_2,s_3 \rangle \simeq \ZZ_2^3$.

For the representation $\rho_H$ of $G$ induced by the trivial representation of $H$ we have
$$
\rho_H = \chi_0 \oplus W_3 \oplus \psi
$$
with $\psi = \rho_1 \oplus \ldots \oplus \rho_4$ is the rational irreducible representation of degree four given by the sum of the linear complex irreducible characters of $G$ induced by each of the non-trivial characters of $P \simeq \ZZ_5$ by the 
projection $G \ra P$. Hence the assumptions 
\eqref{e4.1} are satisfied for the representation $W_3$.

 The double coset decomposition of $G$ is
$$
G = H \cup (Hs_4H) \cup (H\sigma^4H) \cup (H\sigma^3H) ) \cup (H\sigma^2H) \cup (H\sigma H)=: H_1 \cup  \cdots  \cup H_6.
$$
where the first double coset is $H$, the second is $N - H$, and the other four consist of 16 elements of order five each.

A basis for the (commutative) Hecke algebra $\cH_H$ is given by
$$
q_1 = p_{H} \; \; \mbox{and} \;\; q_j  = \frac{1}{8} \sum_{k_j \in H_j} k_j \;\; \mbox{for} \;\; j = 2, \dots,6
$$
and one checks
$$
q_1 q_j = q_j \;\; \mbox{for all} \;\; j, \quad q_2q_j = 8 q_j \;\; \mbox{for} \;\; 3 \leq l \le 6,
$$
$$
q_2^2 = 64q_1, \;\; q_3^2 = 8 q_3, \;\; q_4^2 = 8q_6, \;\; q_5^2 = 8 q_3, \;\; q_6^2 = 8 q_5.
$$
Since 
$$
\chi_{W_3}(q_1) = 1 , \chi_{W_3}(q_2) = -1 \;\;\mbox{and} \;\; \chi_{W_3}(q_i) = 0 \;\; \mbox{for} \; i = 3, \dots,6,
$$
we get from Theorem \ref{thm4.3},
$$
A_{H,\widetilde W_3} = \{ z \in A_H \;|\; q_2(z) = -z, q_i(z) = 0 \; \mbox{for} \; i = 3,\dots , 6 \}_0 = P(A_H/A_N).
$$
An analogous result can be proved for the representations $\widetilde W_1$ and $\widetilde W_2$.


\begin{thebibliography}{999999}
 \bibitem{bl}
Ch. Birkenhake, H. Lange,  {\it Complex Abelian Varieties}. 2nd edition, Grundl. Math. Wiss. vol 302, Springer (2004).
\bibitem{clr}
A. Carocca, H. Lange, R. E. Rodr{\'i}guez: {\it \'Etale double covers of cyclic $p$-gonal curves},  Preprint (2018). 
\bibitem{caro0}
A. Carocca, R. E. Rodr{\'i}guez, {\it Jacobians with group actions and rational idempotents}. Journ. of Alg. 306 (2006), 322-343.
\bibitem{caro}
A. Carocca, R. E. Rodr{\'i}guez, {\it Hecke algebras acting on abelian varieties}. J. Pure Appl. Algebra 222 (2018), 2626-2647.
\bibitem{cr}
C.W. Curtis, I. Reiner, {\it Methods of representation theory, Vol. 1}. John Wiley, (1981).
\bibitem{h}
M. Hall, {\it Combinatorial Theory}. 2nd ed., John Wiley, (1986).
\bibitem{i} 
I. M. Isaacs, {\it Character theory of finite groups}. Academic Press, (1976).
\bibitem{in}
I. M. Isaacs, G. Navarro: {\it Character sums and double cosets}. Journ. of Alg. 320 (2008), 3749-3764. 
\bibitem{rr}
 S. Recillas, R. E. Rodr{\'i}guez: {\it Prym varieties and fourfold covers}. arXiv:math/0303155 .
\end{thebibliography}
\end{document}